\documentclass[eqnumis, pdf]{yjb}
\usepackage{amsfonts}
\usepackage{amsmath}
\usepackage{footnote}
\usepackage{multicol}
\usepackage{booktabs}
\usepackage[flushleft]{threeparttable}
\usepackage[font=small,labelfont=bf]{caption}
\usepackage{float}
\usepackage{hyperref}

\hypersetup{
colorlinks,
citecolor=magenta,
linkcolor=blue,
urlcolor=magenta}
\usepackage{mathrsfs,upref}

\begin{document}
%please do not change this part%%%%%

%%%%%%%%%%%%%%%%%%%%%%%%%%%%%%
%\giris
%{ } %yazarlar
%{ } %başlık
%{ } %ilk sayfa no
%author, title, first page number
%%%%%%%%%%%%%%%%%%%%%%%%%%%%%%
%Corresponding Author Email:
% 
%%%%%%%%%%%%%%%%%%%%%%%%%%%%%%

\markboth{Yogesh J. Bagul}{Stringent bounds for the non-zero Bernoulli numbers}

\title{Stringent bounds for the non-zero Bernoulli numbers}

\author{Yogesh J. Bagul}

\address{Department of Mathematics\\ K. K. M. College, Manwath\\ Dist: Parbhani (M. S.) - 431505, India
	}
\email{yjbagul@gmail.com}

\maketitle
\begin{abstract}
We present new sharper lower and upper bounds for the non-zero Bernoulli numbers using Euler's formula for the Riemann zeta function. In particular, we determine the best possible constants $ \alpha $ and $ \beta $ such that the double inequality
$$ \frac{2\cdot (2k)!}{\pi^{2k} (2^{2k}-1)}\frac{3^{2k}}{(3^{2k}-\alpha)} < \vert B_{2k} \vert < \frac{2\cdot (2k)!}{\pi^{2k} (2^{2k}-1)}\frac{3^{2k}}{(3^{2k}-\beta)}, $$
holds for $ k = 1, 2, 3, \cdots.$ Our main results refine the existing bounds of $ \vert B_{2k} \vert $ in the literature.
\end{abstract}% the abstract
  % AMS subject classifications
        % Keywords

%please do not change this part%%%%%%%%%%%receive, accept
%%%%%%%%%%%%%%%%%%%%%%%%%%%%%%%%%%%

%Accepted Date: 
%Received Date: 
%DOI: 
\vspace{10pt}

\section{Introduction}\label{sec1}
The classical Bernoulli numbers frequently occur in mathematical analysis and other branches of mathematics and science. These fascinating numbers which are a sequence of rational numbers were discovered by the Swiss mathematician Jacob Bernoulli. They are denoted by $ B_k $ and may be defined by (\cite[p. 804]{abramowitz}, \cite[p. 3]{temme})
$$ \frac{x}{e^x - 1} = \sum_{k=0}^\infty B_k \frac{x^k}{k!},  \, \, \vert x \vert < 2\pi. $$
 On the one hand, all the odd-indexed Bernoulli numbers $B_{2k+1} , \, \,  k \in \mathbb{N}$ are zero except $ B_1 = -1/2$. On the other hand, all the even-indexed Bernoulli numbers $B_{2k}$ are non-zero and they alternate in signs, i.e., for $ k= 1, 2, 3, 4, 5, 6, 7, 8, \cdots$, the numbers $ B_{2k} $ are respectively given by $1/6, -1/30, 1/42, -1/30, 5/66, -691/2730, 7/6, -3617/510, \cdots.$ Hence, $$ \vert B_{2k} \vert = (-1)^{k+1} B_{2k}, \, \, k = 1, 2, 3, \cdots. $$
The bounds for even-indexed Bernoulli numbers play a vital role in the theory of inequalities and thus bounding these Bernoulli numbers has been the topic of interest.

The double inequality 
\begin{align}\label{eqn1.1}
\frac{2 \cdot(2k)!}{(2 \pi)^{2k}} < \vert B_{2k} \vert < \frac{2 \cdot(2k)!}{(2 \pi)^{2k}} \cdot \frac{2^{2k-1}}{(2^{2k-1}-1)} = \frac{(2k)!}{\pi^{2k}(2^{2k-1}-1)}; \, \, k = 1, 2, \cdots
\end{align}
 appeared in \cite[p. 805; 23.1.15]{abramowitz}. In 1989, D. J. Leeming \cite[p. 128]{leeming} established that 
\begin{align}\label{eqn1.2}
4 \sqrt{\pi k} \left( \frac{k}{\pi e} \right)^{2k} < \vert B_{2k} \vert < 5 \sqrt{\pi k}\left( \frac{k}{\pi e} \right)^{2k}; \, \, k \geq 2.
\end{align}
As stated in \cite{daniello} the left inequality of (\ref{eqn1.2}) was already established by A. Laforgia \cite[p. 2]{laforgia} in 1980. By utilizing Fourier representation, 
\begin{align}\label{eqn1.3}
B_{2k} = \frac{(-1)^{k-1} \cdot 2 \cdot (2k)!}{(2\pi)^{2k}} \sum_{n=1}^\infty \frac{1}{n^{2k}},
\end{align}
 C. D'Aniello \cite{daniello} obtained 
\begin{align}\label{eqn1.4}
\frac{2 \cdot(2k)!}{(2 \pi)^{2k}} \left( 1 + \frac{1}{2^{2k}} \right) < \vert B_{2k} \vert ; \, \, k = 1, 2, \cdots.
\end{align}
Further, using the Stirling formula \cite{abramowitz}, the inequality 
\begin{align}\label{eqn1.5}
4 \sqrt{\pi k} \left( \frac{k}{\pi e} \right)^{2k} \left( 1 + \frac{1}{2^{2k}} \right) < \vert B_{2k} \vert ; \, \, k = 1, 2, \cdots
\end{align}
was also obtained in \cite{daniello}. In the same paper \cite{daniello}, a relation \cite[p. 807]{abramowitz}
\begin{align}\label{eqn1.6}
\vert B_{2k} \vert = \frac{2 \cdot (2k)!}{\pi^{2k}} \frac{1}{(2^{2k}-1)}\sum_{n=0}^\infty \frac{1}{(2n+1)^{2k}} ; \, \, k = 1, 2, \cdots
\end{align}
was used to establish 
\begin{align}\label{eqn1.7}
\frac{2 \cdot (2k)!}{\pi^{2k}} \frac{1}{(2^{2k}-1)} < \vert B_{2k} \vert ; \, \, k = 1, 2, \cdots.
\end{align}
Among all the lower bounds of $ \vert B_{2k} \vert $ listed above, the lower bound in (\ref{eqn1.7}) is stringent. Thus, by combining (\ref{eqn1.1}) and (\ref{eqn1.7}) we have 
\begin{align}\label{eqn1.8}
\frac{2 \cdot(2k)!}{\pi^{2k}(2^{2k}-1)} < \vert B_{2k} \vert <  \frac{(2k)!}{\pi^{2k}(2^{2k-1}-1)}; \, \, k = 1, 2, \cdots.
\end{align}
For $ k \in \mathbb{N} $ (the set of positive integers), H. Alzer \cite{alzer} showed that the constants $ \theta = 0 $ and $ \delta = 2 + \frac{\ln(1-6/\pi^2)}{\ln 2} $ such that 
\begin{align}\label{eqn1.9}
\frac{2 \cdot (2k)!}{(2\pi)^{2k}} \frac{1}{1-2^{\theta - 2k}} \leq \vert B_{2k} \vert \leq \frac{2 \cdot (2k)!}{(2\pi)^{2k}} \frac{1}{1-2^{\delta - 2k}}
\end{align}
are the best possible. The upper bound in (\ref{eqn1.9}) is sharper than that in (\ref{eqn1.8}). Another upper bound for $ \vert B_{2k} \vert $ was given by H.-F. Ge \cite{ge} as follows:
\begin{align}\label{eqn1.10}
\vert B_{2k} \vert \leq \frac{2(2^{2n}-1)}{2^{2n}} \zeta(2n) \frac{(2k)!}{\pi^{2k}(2^{2k}-1)}; \, \,  k \geq n \, \, \text{and} \, \, n \in \mathbb{N},
\end{align}
where $ \zeta $ is the Riemann zeta function.

 The upper bound in (\ref{eqn1.8}) has been nicely sharpened, viz. in (\ref{eqn1.9}). However, to the best of the author's knowledge and belief, the lower bound in (\ref{eqn1.8}) is not presented in its refined form. In fact, from (\ref{eqn1.3}) and (\ref{eqn1.6}) we immediately have the following proposition.
\begin{proposition}\label{prop1.1}
\begin{itemize}
\item[(i)] For $ m, k \in \mathbb{N}, $ we have
\begin{align}\label{eqn1.11}
\frac{ 2 \cdot (2k)!}{(2\pi)^{2k}} \sum_{n=1}^m \frac{1}{n^{2k}} < \vert B_{2k} \vert.
\end{align} 
\item[(ii)] For $ m = 0, 1, 2, \cdots $ and $ k \in \mathbb{N} $ we have 
\begin{align}\label{eqn1.12}
\frac{2 \cdot (2k)!}{\pi^{2k}} \frac{1}{(2^{2k}-1)}\sum_{n=0}^m \frac{1}{(2n+1)^{2k}} <  \vert B_{2k} \vert.
\end{align}
\end{itemize} 
\end{proposition}  
It is not difficult to prove that the lower bound in (\ref{eqn1.12}) is finer than that in (\ref{eqn1.11}) for $ m \geq 1. $ Putting $ m = 2 $ and $ m = 0 $ in (\ref{eqn1.11}) and (\ref{eqn1.12}) respectively we get the inequalities (\ref{eqn1.4}) and (\ref{eqn1.7}). To obtain sharper lower bounds, one should increase the value of $ m $ in (\ref{eqn1.11}) and (\ref{eqn1.12}). For example, by putting $ m=3 $ and $ m = 1$ in (\ref{eqn1.11}) and (\ref{eqn1.12}) respectively we get
\begin{align}\label{eqn1.13}
\frac{ 2 \cdot (2k)!}{(2\pi)^{2k}} \left(1 + \frac{1}{2^{2k}} + \frac{1}{3^{2k}}\right) < \vert B_{2k} \vert ; \, \, k = 1, 2, \cdots
\end{align}
and 
\begin{align}\label{eqn1.14}
\frac{2 \cdot (2k)!}{\pi^{2k}} \frac{1}{(2^{2k}-1)} \left( 1 + \frac{1}{3^{2k}} \right) <  \vert B_{2k} \vert ; \, \, k = 1, 2, \cdots.
\end{align}

Despite this, we will present more stringent and convincing lower bounds for $ \vert B_{2k} \vert $ in the next section. Our main aim is to refine the double inequality (\ref{eqn1.9}) by establishing sharp bounds for $ \vert B_{2k} \vert .$ 

\section{Main results}\label{sec2}
Let us first define $ \mathcal{P}_0 = \left\lbrace p_0 = 2, p_1 = 3, p_2 = 5, p_3 = 7, p_4 = 11, p_5 = 13, p_6 = 17, \cdots \right\rbrace, $ as the set of all positive prime integers and $ \mathcal{P} = \left\lbrace p_1 = 3, p_2 = 5, p_3 = 7, p_4 = 11, p_5 = 13, p_6 = 17, \cdots \right\rbrace $ as the set of all positive odd prime integers. 

We start with the following basic inequality which will motivate us to establish the main result of the paper.
\begin{proposition}\label{prop2.1}
The inequality  
\begin{align}\label{eqn2.1}
\frac{2 \cdot (2k)!}{\pi^{2k} (2^{2k}-1)} \prod_{n=1}^m \left(\frac{p_n^{2k}}{p_n^{2k}-1}\right) < \vert B_{2k} \vert , 
\end{align}
holds for $m, k \in \mathbb{N} $ and $p_n \in \mathcal{P}$.  
\end{proposition}
\begin{proof}
Using Euler's fabulous product formula \cite{gautschi} for the zeta function $\zeta$, we have 
$$ \zeta(2k) = \left(\frac{2^{2k}}{2^{2k}-1}\right) \prod_{n \in \mathbb{N}} \left(\frac{p_n^{2k}}{p_n^{2k}-1}\right), $$ where $ k \in \mathbb{N}, \, \, p_n \in \mathcal{P}. $ This implies 
$$ \zeta(2k) > \left(\frac{2^{2k}}{2^{2k}-1}\right) \prod_{n=1}^m \left(\frac{p_n^{2k}}{p_n^{2k}-1}\right). $$
It is given in \cite[p. 5, (1.14)]{temme} that 
$$ \vert B_{2k} \vert = \frac{2 \cdot (2k)!}{(2\pi)^{2k}} \zeta(2k), \, \, k \in \mathbb{N}. $$ Hence
$$ \frac{\pi^{2k}(2^{2k} - 1)}{2 \cdot (2k)!} \vert B_{2k} \vert > \prod_{n=1}^m \left(\frac{p_n^{2k}}{p_n^{2k}-1}\right), $$ which gives the desired result.
\end{proof}
In particular, for $ m = 1, 2, $ Proposition \ref{prop2.1} yields respectively the following inequalities:
\begin{align}\label{eqn2.2}
\frac{2 \cdot (2k)!}{\pi^{2k} (2^{2k} - 1)} \frac{3^{2k}}{(3^{2k}-1)} < \vert B_{2k} \vert , \, \, k \in \mathbb{N}
\end{align}
and
\begin{align}\label{eqn2.3}
\frac{2 \cdot (2k)!}{\pi^{2k} (2^{2k} - 1)} \frac{3^{2k}}{(3^{2k}-1)} \frac{5^{2k}}{(5^{2k}-1)} < \vert B_{2k} \vert , \, \, k \in \mathbb{N}.
\end{align}
It is easy to check that the inequalities (\ref{eqn2.2}) and (\ref{eqn2.3}) are refinements of the inequality (\ref{eqn1.14}). In Proposition \ref{prop2.1}, we also observe that the larger the value of $ m $, the sharper is the corresponding inequality.

\begin{remark} The statement of Proposition \ref{prop2.1} can be reformulated as follows: 

 If $ m $ is a non-negative integer, then 
\begin{align}\label{eqn2.4}
\frac{2 \cdot (2k)!}{(2\pi)^{2k}} \prod_{n=0}^m \left(\frac{p_n^{2k}}{p_n^{2k}-1}\right) < \vert B_{2k} \vert , 
\end{align}
where $k \in \mathbb{N} $ and $p_n \in \mathcal{P}_0.$  
\end{remark}
Here, (\ref{eqn2.4}) clearly generalizes the inequality (\ref{eqn1.7}). \\

The next Proposition shows that the lower bounds in Proposition \ref{prop2.1} are sharper than those in Proposition \ref{prop1.1}.
\begin{proposition}\label{prop2.2}
For $k \in \mathbb{N} $ and $p_n \in \mathcal{P},$ it is true that
$$ \prod_{n=1}^m \left(\frac{p_n^{2k}}{p_n^{2k}-1}\right) > \sum_{n=0}^m \frac{1}{(2n+1)^{2k}}, $$  
\end{proposition}
\begin{proof}
For $k \in \mathbb{N} $ and $p_n \in \mathcal{P}$, we write
$$ \prod_{n=1}^m \left(\frac{p_n^{2k}}{p_n^{2k}-1}\right) = \prod_{n=1}^m \left(\frac{1}{1-\frac{1}{p_n^{2k}}}\right), $$ i.e.,
$$\prod_{n=1}^m \left(\frac{p_n^{2k}}{p_n^{2k}-1}\right) = \left(\frac{1}{1-\frac{1}{3^{2k}}}\right) \cdot \left(\frac{1}{1-\frac{1}{5^{2k}}}\right) \cdot \left(\frac{1}{1-\frac{1}{7^{2k}}}\right) \cdot \left(\frac{1}{1-\frac{1}{11^{2k}}}\right) \cdots \left(\frac{1}{1-\frac{1}{p_m^{2k}}}\right).$$ Making use of $$ \frac{1}{1-x} = 1 + x + x^2 + x^3 + \cdots, \, \, \vert x \vert < 1, $$ we get
\begin{multline}\nonumber
\prod_{n=1}^m \left(\frac{p_n^{2k}}{p_n^{2k}-1}\right) = \left( 1 + \frac{1}{3^{2k}} + \frac{1}{9^{2k}} + \frac{1}{3^{6k}} + \cdots \right) \times \left( 1 + \frac{1}{5^{2k}} + \frac{1}{5^{4k}} + \frac{1}{5^{6k}} + \cdots \right) \times  \\ 
\left( 1 + \frac{1}{7^{2k}} + \frac{1}{7^{4k}} + \frac{1}{7^{6k}} + \cdots \right) \times \left( 1 + \frac{1}{11^{2k}} + \frac{1}{11^{4k}} + \frac{1}{11^{6k}} + \cdots \right) \times \\
\cdots \left( 1 + \frac{1}{p_m^{2k}} + \frac{1}{p_m^{4k}} + \frac{1}{p_m^{6k}} + \cdots \right).\\
\end{multline}
Thus 
\begin{align*}
\prod_{n=1}^m \left(\frac{p_n^{2k}}{p_n^{2k}-1}\right) = &\left(1 + \frac{1}{3^{2k}} + \frac{1}{5^{2k}} + \frac{1}{7^{2k}} + \frac{1}{9^{2k}} \cdots + \frac{1}{(2m+1)^{2k}} + \cdots + \frac{1}{p_m^{2k}}\right) \\
&+ \text{sum of numbers greater than zero} \\
&> 1 + \frac{1}{3^{2k}} + \frac{1}{5^{2k}} + \frac{1}{7^{2k}} + \frac{1}{9^{2k}} \cdots + \frac{1}{(2m+1)^{2k}} = \sum_{n=0}^m \frac{1}{(2n+1)^{2k}}.
\end{align*}
The proof is completed.
\end{proof}
Let $L_{2k} $ be the lower bound of $ \vert B_{2k} \vert $ established in (\ref{eqn2.1}). For $ m = 1, $ we list the approximate numerical values $ \vert B_{2k} \vert - L_{2k}$ in the below table.

{
\begin{table}[H]
\centering
\caption{ Numerical values of $ \vert B_{2k} \vert - L_{2k}$ for $ m = 1.$ }
\vspace{.5cm}
{\renewcommand{\arraystretch}{2}
{\small
\begin{tabular}{c|c|c|c|c|c|c|c|}
\cline{2-6}
& \multicolumn{5}{|c|}{$ m = 1 $}  \\ \hline
\multicolumn{1}{|c|}{$k$} &$ 1 $ &$ 2  $ &$ 3$ &$4 $ & $5$  \\\hline
\multicolumn{1}{|c|}{$\vert B_{2k} \vert - L_{2k}$} & $ 1.46... \times 10^{-2}$  &$ 7.15... \times 10^{-5} $  &$  1.74... \times 10^{-6} $  &$  9.13... \times 10^{-8}  $ & $ 8.02... \times 10^{-9}$   \\ \hline
\cline{2-6}
& \multicolumn{4}{|c|}{$ m = 1$}    \\ \hline
\multicolumn{1}{|c|}{$k$} &$6$ &$7$ &$ 8$ &$9 $ & $10$  \\
\hline
\multicolumn{1}{|c|}{$\vert B_{2k} \vert - L_{2k}$} & $1.05... \times 10^{-9}$  &$ 1.92... \times 10^{-10} $  &$ 4.66... \times 10^{-11} $  &$ 1.43... \times 10^{-11} $ & $5.11... \times 10^{-12}$   \\
\hline
\end{tabular}
}}
\label{tab1}
\end{table}
}

From these values, we infer that the lower bound in (\ref{eqn2.1}) is sufficiently sharp even for $ m=1.$ One should get sharper and sharper lower bounds for larger values of $m.$

\begin{corollary}\label{cor1}
Suppose that $ a $ is any real number and $ k \in \mathbb{N}. $ Then 
\begin{align}\label{eqn2.5}
\frac{2 \cdot (2k)!}{\pi^{2k} (2^{2k} - 1)} \frac{a^{2k}}{(a^{2k}-1)} < \vert B_{2k} \vert
\end{align}
if $ a \geq 3.$
\end{corollary}
\begin{proof}
For $ a \geq 3, $ it is obvious that
$$ \frac{a^{2k}}{(a^{2k}-1)} \leq \frac{3^{2k}}{(3^{2k}-1)}. $$ Then by (\ref{eqn2.2}), we get the inequality (\ref{eqn2.5}).
\end{proof}
\begin{remark}
Applying the same argument as in the proof of Corollary \ref{cor1}, for any real number $ a $ and $ m, k \in \mathbb{N}, $ we have 
\begin{align}\label{eqn2.6}
\frac{2 \cdot (2k)!}{\pi^{2k}(2^{2k}-1)}\frac{3^{2k}}{(3^{2k}-1)}\frac{5^{2k}}{(5^{2k}-1)} \cdots \frac{p_{m-1}^{2k}}{(p_{m-1}^{2k}-1)} \frac{a^{2k}}{(a^{2k}-1)} < \vert B_{2k} \vert
\end{align}
if $ a \geq p_m, $ where $ p_i \in \mathcal{P}.$
\end{remark}
Table \ref{tab1} shows that the inequality (\ref{eqn2.2}) is very sharp. It is of particular interest and it motivates us to ask the natural question: What can be the best possible constants $ \alpha $ and $ \beta $ such that the double inequality 
$$ \frac{2\cdot (2k)!}{\pi^{2k} (2^{2k}-1)}\frac{3^{2k}}{(3^{2k}-\alpha)} < \vert B_{2k} \vert < \frac{2\cdot (2k)!}{\pi^{2k} (2^{2k}-1)}\frac{3^{2k}}{(3^{2k}-\beta)}, $$
holds for $ k = 1, 2, 3, \cdots?$

The answer to the above question is provided in the following theorem which is our main finding.

\begin{theorem}\label{thm2.6}
The best possible constants $ \alpha $ and $ \beta $ satisfying the double inequality
\begin{align}\label{eqn2.7}
\frac{2\cdot (2k)!}{\pi^{2k} (2^{2k}-1)}\frac{3^{2k}}{(3^{2k}-\alpha)} < \vert B_{2k} \vert < \frac{2\cdot (2k)!}{\pi^{2k} (2^{2k}-1)}\frac{3^{2k}}{(3^{2k}-\beta)}, \, \, k \in \mathbb{N}
\end{align}
are $ 1 $ and $ 9\left(1 - \frac{8}{\pi^2}\right) \approx 1.704875 $ respectively.
\end{theorem}
\begin{proof}
Utilizing the same relation
$$ \zeta(2k) = \frac{(2\pi)^{2k}}{2(2k)!} \vert B_{2k} \vert , \, \, k \in \mathbb{N} $$ as in the proof of Proposition \ref{prop2.1}, the double inequality (\ref{eqn2.7}) can be equivalently written as
$$ \frac{2^{2k}}{(2^{2k}-1)}\frac{3^{2k}}{(3^{2k}-\alpha)} < \zeta(2k) < \frac{2^{2k}}{(2^{2k}-1)}\frac{3^{2k}}{(3^{2k}-\beta)}, $$ i.e.,
$$ \alpha < 3^{2k}\left[1 - \frac{1}{\left(1-\frac{1}{2^{2k}}\right)\zeta(2k)}\right] < \beta , \, \, k \in \mathbb{N}. $$ 
We define $$ h(x) = 3^x \left[1 - \frac{1}{\left(1-\frac{1}{2^x}\right)\zeta(x)}\right], \, \, x = \left\lbrace 2, 4, 6, 8, \cdots \right\rbrace. $$ However, we prove the monotonicity of $ h(x) $ for $ x \in \mathbb{N} - \left\lbrace 1 \right\rbrace.$ Consider $ g(x) = h(x) - h(x+1).$ Then 
\begin{align*}
g(x) &= 3^x\left[1- \frac{1}{\left(1-\frac{1}{2^x}\right)\zeta(x)}\right] - 3^{x+1}\left[1- \frac{1}{\left(1-\frac{1}{2^{x+1}}\right)\zeta(x+1)}\right] := 3^x \cdot t(x),
\end{align*}
where
\begin{align*}
t(x) &= \frac{3}{\left(1-\frac{1}{2^{x+1}}\right)\zeta(x+1)} - \frac{1}{\left(1-\frac{1}{2^x}\right)\zeta(x)} - 2 \\
&= \frac{3\left(1-\frac{1}{2^x}\right)\zeta(x) - \left(1-\frac{1}{2^{x+1}}\right)\zeta(x+1) }{\left(1-\frac{1}{2^{x+1}}\right)\zeta(x+1) \cdot \left(1-\frac{1}{2^x}\right)\zeta(x)} - 2.
\end{align*}
Now using Euler's thinking \cite[p. 15]{gautschi}, we have $$ \left(1-\frac{1}{2^x}\right)\zeta(x) = 1 + \frac{1}{3^x} + \frac{1}{5^x} + \frac{1}{7^x} + \frac{1}{9^x} + \cdots = \sum_{n=0}^\infty \frac{1}{(2n+1)^x}.  $$
Therefore
\begin{align*}
t(x) &= \frac{3 \sum_{n=0}^\infty \left(\frac{1}{2n+1}\right)^x - \sum_{n=0}^\infty \left(\frac{1}{2n+1}\right)^{x+1}}{\left[\sum_{n=0}^\infty \left(\frac{1}{2n+1}\right)^{x+1}\right] \left[ \sum_{n=0}^\infty \left(\frac{1}{2n+1}\right)^x \right]} - 2 \\
&= \frac{2 \sum_{n=0}^\infty \left(\frac{3n+1}{2n+1}\right) \left(\frac{1}{2n+1}\right)^x}{\sum_{n=0}^\infty\sum_{l=0}^n  \left(\frac{1}{2l+1}\right)^{x+1} \left(\frac{1}{2n-2l+1}\right)^x } - 2 \\
&:= \frac{2 \sum a_n}{\sum b_n} - 2.
\end{align*} 
The actual terms of the series $ \sum b_n$ are given by
\begin{multline}\nonumber
\sum b_n = 1 + \frac{(4/3)}{3^x} + \frac{(6/5)}{5^x} + \frac{(8/7)}{7^x} + \frac{(13/9)}{9^x} + \frac{(12/11)}{11^x} + \frac{(14/13)}{13^x} \\
+ \frac{(24/15)}{15^x} + \frac{(18/17)}{17^x} + \frac{(20/19)}{19^x} + \frac{(32/21)}{21^x} + \cdots.
\end{multline}
Similarly, the terms of the series $ \sum a_n $ can be written as
\begin{multline}\nonumber
\sum a_n = 1 + \frac{(4/3)}{3^x} + \left[\frac{(6/5)}{5^x} + \frac{(1/5)}{5^x}\right] + \left[\frac{(8/7)}{7^x} + \frac{(2/7)}{7^x}\right] + \frac{(13/9)}{9^x} + \left[\frac{(12/11)}{11^x} + \frac{(4/11)}{11^x}\right] \\ + \left[\frac{(14/13)}{13^x} + \frac{(5/13)}{13^x}\right] + \left[\frac{(24/15)}{15^x} - \frac{(2/15)}{15^x}\right] + \left[\frac{(18/17)}{17^x} + \frac{(7/17)}{17^x}\right] \\
 + \left[\frac{(20/19)}{19^x} + \frac{(8/19)}{19^x}\right] + \left[\frac{(32/21)}{21^x} - \frac{(1/21)}{21^x}\right] + \cdots,
\end{multline}
i.e.,
\begin{align*}
\sum a_n = &\sum b_n \\
&+ \left[\frac{(1/5)}{5^x} + \frac{(2/7)}{7^x} + \frac{(4/11)}{11^x} + \frac{(5/13)}{13^x} - \frac{(2/15)}{15^x} + \frac{(7/17)}{17^x} + \frac{(8/19)}{19^x} - \frac{(1/21)}{21^x} + \cdots \right].
\end{align*}
From this, it is obvious that $ \sum a_n > \sum b_n. $ So we get $ t(x) > 0$ and hence $ g(x) > 0, $ which implies that $ h(x) $ is strictly decreasing for $ x \in \mathbb{N} - \left\lbrace 1 \right\rbrace.$ Thus, $ h(x) $ is strictly decreasing for $ x \in \left\lbrace 2, 4, 6, \cdots \right\rbrace.$ Consequently,
$ \alpha = \lim_{x \rightarrow \infty^-} h(x) = 1 $ and $ \beta = h(2) = 9\left(1-\frac{8}{\pi^2}\right). $ This completes the proof.
\end{proof}
\begin{remark}\label{remI}
 The lower bound in (\ref{eqn2.7}) is clearly stronger than that in (\ref{eqn1.9}). The upper bound in (\ref{eqn2.7}) is also stronger than that in (\ref{eqn1.9}). Because
 \begin{align*}\nonumber
&\frac{2\cdot (2k)!}{\pi^{2k} (2^{2k}-1)}\frac{3^{2k}}{(3^{2k}-\beta)} \leq \frac{2 \cdot (2k)!}{(2\pi)^{2k}} \frac{1}{1-2^{\delta - 2k}} \\
&\Longleftrightarrow \frac{1}{(2^{2k}-1)}\frac{3^{2k}}{(3^{2k}-\beta)} \leq \frac{1}{(2^{2k}-2^\delta)} \\
&\Longleftrightarrow \beta \cdot (2^{2k}-1) \leq (2^\delta - 1) \cdot 3^{2k}.
 \end{align*}
 The last inequality holds for all $ k \in \mathbb{N} $ due to 
 $$ \frac{\beta}{(2^\delta - 1)} = \frac{9\left(1-\frac{8}{\pi^2}\right)}{4 \cdot 2^{\frac{\ln (1-6/\pi^2)}{\ln 2}}-1} = \frac{9\left(1-\frac{8}{\pi^2}\right)}{4\left(1-\frac{6}{\pi^2}\right)-1} = 3.$$
\end{remark}

{%Suppose that $ U_{2k} $ be the upper bound of $ \vert B_{2k} \vert $ as proved in (\ref{eqn2.7}). The sharpness of $ U_{2k} $ can be seen from the following table.}

If we set
\begin{align*}
E_1(k) &= \frac{2 \cdot (2k)!}{\pi^{2k}(2^{2k}-1)}\frac{3^{2k}}{(3^{2k}-\alpha)} - \frac{2 \cdot (2k)!}{(2\pi)^{2k}(1-2^{-2k})} \\
 &= \frac{2 \cdot (2k)!}{\pi^{2k}(2^{2k}-1)} \left(\frac{3^{2k}}{3^{2k}-1} - 1 \right)
\end{align*} 
and 
\begin{align*}
E_2(k) &= \frac{2 \cdot (2k)!}{(2\pi)^{2k}} \frac{1}{1-2^{\delta - 2k}} - \frac{2\cdot (2k)!}{\pi^{2k} (2^{2k}-1)}\frac{3^{2k}}{(3^{2k}-\beta)} \\
&= \frac{2 \cdot (2k)!}{\pi^{2k}} \left( \frac{1}{(2^{2k}-2^\delta)} - \frac{3^{2k}}{(2^{2k}-1)(3^{2k}-\beta)}\right),
\end{align*}
then the validity of our claim in Remark \ref{remI} is also demonstrated in Figure \ref{fig1}.

\begin{figure}[H]
\centering
\includegraphics[width=10cm,height=7cm]{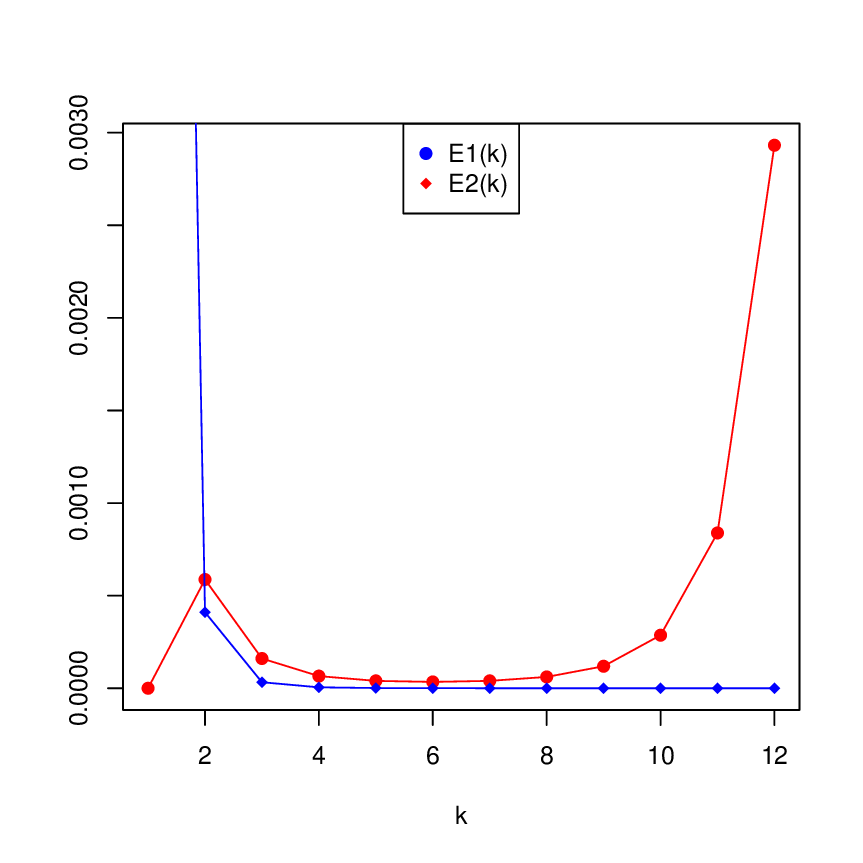}  
\caption{Plots of $ E_1(k) $ and $ E_2(k) $ showing the difference between corresponding lower and upper bounds of $ \vert B_{2k} \vert $ in (\ref{eqn1.9}) and (\ref{eqn2.7}) for $ k = 1, 2, 3, \cdots 12.$}\label{fig1}
\end{figure}

Lastly, it is not difficult to show or verify that an upper bound in (\ref{eqn2.7}) is stronger than that in (\ref{eqn1.10}).

\section{Conclusion}
In this paper, several bounds for the non-zero Bernoulli numbers have been obtained. The double inequality (\ref{eqn2.7}) is of particular interest and it improves all the existing bounds of non-zero Bernoulli numbers in the literature. Due to the appearance of the Bernoulli numbers in the series expansions of trigonometric and hyperbolic functions, the author believes that inequality (\ref{eqn2.7}) will play a crucial role in obtaining fine inequalities involving these functions. Further, it should be noted that a function $ h(x) $ in thye proof of Theorem \ref{thm2.6} can be shown strictly decreasing for all real numbers $ x \geq 2 $ by another method which may be applied to prove the following conjecture.\\

{\bf Conjecture 1.}\label{conj1}
The double inequality
\begin{align*}
 \prod_{n=1}^{m-1} \left(\frac{p_n^{2k}}{p_n^{2k}-1}\right) \times  \frac{p_m^{2k}}{(p_m^{2k} - \alpha')}  < \frac{\pi^{2k} (2^{2k}-1)}{2\cdot (2k)!} \vert B_{2k} \vert 
<  \prod_{n=1}^{m-1} \left(\frac{p_n^{2k}}{p_n^{2k}-1}\right) \times  \frac{p_m^{2k}}{(p_m^{2k} - \beta')}, 
\end{align*}
where $m, k \in \mathbb{N}, \, m \geq 2 $ and $p_n \in \mathcal{P}$ holds with the best possible constants  $ \alpha' = 1 $ and $ \beta' = p_m^2\left[1-\frac{8}{\pi^2} \prod_{n=1}^{m-1}\frac{p_n^2}{(p_n^2-1)} \right]. $ \\

{\bf Note.} This is the updated version of this manuscript on arXiv, and all previous versions should be disregarded.

\end{document}